\documentclass [reqno,10pt,oneside]{amsart}

\usepackage{caption}
 \usepackage{subcaption}
\usepackage{amsthm}
\usepackage{amscd}
\usepackage{amsmath,amssymb}
\usepackage{algorithm}
\usepackage{placeins}
\usepackage{graphicx}
\usepackage{mathtools}
\usepackage{placeins}
\usepackage{float}
\usepackage{epsfig}
\usepackage{epstopdf}
\usepackage[noend]{algorithmic}
\usepackage[utf8]{inputenc}
\restylefloat{table}

\theoremstyle {plain}
\newtheorem {thm}{Theorem}[section]
\newtheorem {prop}[thm]{Proposition}
\newtheorem {lem}[thm]{Lemma}
\newtheorem {cor}[thm]{Corollary}
\theoremstyle {definition}

\theoremstyle {remark}

\begin{document}

\bibliographystyle{alpha}

\title{Classification of Planar Graphs Associated to the Ideal of the Numerical Semigroup}

\author{Muhammad Ahsan Binyamin, Wajid Ali, Adnan Aslam, Hasan Mahmood}


\address{Muhammad Ahsan Binyamin\\ Department of Mathematics, GC University, Faisalabad,  Pakistan}
\email{ahsanbanyamin@gmail.com}

\address{Wajid Ali\\ Department of Mathematics, Department of Mathematics, GC University, Faisalabad,  Pakistan}
\email{wajid1020@yahoo.com }

\address{Adnan Aslam\\ Department of Natural Sciences and Humanities, University of Engineering and Technology, Lahore (RCET) 54000, Pakistan.}
\email{adnanaslam15@yahoo.com}

\address{Hasan Mahmood\\ Department of Mathematics, GC University, Lahore,  Pakistan}
\email{hasanmahmood@gcu.edu.pk}
\keywords{Numerical Semigroup, Complete Graph, Planarity.}
%
%
%
\maketitle

\begin{abstract}
Let $\Lambda$ be a numerical semigroup and $I\subset \Lambda$ be an ideal of $\Lambda$. The graph $G_I(\Lambda)$ assigned to an ideal $I$ of $\Lambda$ is a graph with elements of $(\Lambda \setminus I)^*$ as vertices and any two vertices $x,y$ are adjacent if and only if $x+y \in I$. In this paper we give a complete characterization (up to isomorphism ) of the graph $G_I(\Lambda)$ to be planar, where $I$ is an irreducible ideal of $\Lambda$. This will finally characterize non planar graphs  $G_I(\Lambda)$ corresponding to irreducible ideal $I$.
\end{abstract}
\section{Introduction and Preliminaries}
In the recent years, the study of algebraic structures through the properties of graphs has become an exciting topic of research. This lead to many interesting results and questions. There are many papers on assigning graphs to rings, groups and semigroups see \cite{1,2,3,4,5,6}. Several authors \cite{7,8,9,10,11} studied different properties of these graphs including diameter, girth, domination, central sets and planarity. Let $\Lambda$ be a numerical semigroup. A subset $I\subset \Lambda$ is an ideal (integral ideal) of $\Lambda$ if $I+ \Lambda \subset \Lambda$. The ideal $I$ is irreducible ideal if it cannot be written as intersection of two or more proper ideals which contained it properly. Throughout the paper we consider $I$ to be an irreducible ideal of $\Lambda$. Barucci \cite{12} proved that every irreducible ideal $I$ of numerical semigroup $\Lambda$ can be written in the form $\Lambda \setminus B(x)$, where $B(x)=\{y\in \Lambda:x-y \in \Lambda\}$, for some $x\in \Lambda$. Binyamin et al. \cite{13} assigned a graph to numerical semigroup and studied some properties of this class of graphs. Recently Peng Xu et al. \cite{13a} assigned a graph to the ideal of a numerical semigroup with vertex set $\{v_i\; :\; i\in (\Lambda\setminus I)^*=(\Lambda\setminus I)-\{0\}\}$ and edge set $\{v_iv_j \iff i+j\in I\}$. It is easy to observe that if $I$ is an irreducible ideal of $\Lambda$, then $V(G_I(\Lambda))=\{v_i\; :i\in B^*(x)$\} for some $x\in \Lambda$. For every ideal $I$ of $\Lambda$, the graph $G_I(\Lambda)$ is always connected \cite{13a}. Therefore it is natural to ask when the graph $G_I(\Lambda)$ is complete? It has been proved in \cite{13a}, if $I$ is an irreducible ideal of numerical semigroup $\Lambda$, then the graph $G_I(\Lambda)$ is not a complete graph whenever $|V(G_I(\Lambda))|\geq 3$ and the clique number $ cl(G_I(\Lambda))$ of $\Lambda$ is given by the formula
$$ cl(G_I(\Lambda))=\left\{
                   \begin{array}{ll}
                     \frac{n}{2}+1, & \hbox{if $n$ is even;} \\
                     \frac{n+1}{2}, & \hbox{if $n$ is odd,}
                   \end{array}
                 \right.
$$
where $n$ is the order of the graph $G_I(\Lambda)$. This shows that whenever $n\geq 8$, $G_I(\Lambda)$ has a subgraph isomorphic to complete graph $K_5$ and hence $G_I(\Lambda)$ is not a planar graph. The motivation of this paper is to find all the graphs $G_I(\Lambda)$ that are planar. To answer this question, It is required to give a complete characterization of the graph $G_I(\Lambda)$ such that the order of $G_I(\Lambda)$ is either 6 or 7.

A graph $G=(V(G),E(G))$ is an ordered pair with the vertex set $V(G)$ and the edge set $E(G)$. The cardinality of the vertex set and edge set is called the order and size of $G$ respectively. A graph $G$ is connected if every pair of vertices $x,y\in V(G)$ is connected by a path. A graph $G$ of order $n$ is complete if every pair of vertices of $G$ are adjacent and is denoted by $K_n$. The graph $G$ is bipartite if its vertex set $V(G)$ can be partitioned in to two sets $V_1(G)$ and $V_2(G)$ in such a way that $xy\in E(G)$ if and only if $x\in V_1(G)$ and $y\in V_2(G)$ or vice versa. If  $|V_1(G)|=|V_2(G)|=m$, then $G$ is called a complete bipartite graph and is denoted by $K_{m,m}$. A planar graph is a graph that can be drawn in the plane without crossings that is, no two edges intersect geometrically except at a vertex to which both are incident. Two graphs $G$ and $H$ are said to be homeomorphic if both $G$ and $H$ can be obtained by a same graph by inserting vertices of degree $2$ into its edges. It is well known that $K_{3,3}$ and $K_5$ are non-planar. In order to show that a graph is planar one can use the famous Kuratowski's theorem which states that: A graph is planar if and only if it contains no subgraph homeomorphic to $K_5$ or $K_{3,3}$. For more undefined terminologies related to graph theory see \cite{14,15,16}.

The main aim of this paper is to classify all the graphs $G_I(\Lambda)$ of order 6 and order 7. This, certainly helps us to give a complete answer about the planarity of the graph $G_I(\Lambda)$.

\section{Planar and Non-planar Graphs $G_I(\Lambda)$}
Let $\Lambda=<A>$, where $A=\{a_1,a_2,\dots,a_n\}$ is the minimal system of generators. Let $x\in \Lambda$, then
\[
x=u_1a_1+u_2a_2+\dots+u_na_n,
\]
where $u_1,u_2,\ldots, u_n$ are non negative integers.
Note that $\Lambda$ can be written as
$$\Lambda=\bigcup_{p=1}^nL^{(p)},$$
where $L^{(p)}$ denote the collection of all those elements of $\Lambda$, which can be written as a linear combination of exactly $p$ elements of $A$ that is if $x \in L^{(p)}$ then $x=u_1a_{i_1}+u_2a_{i_2}+\ldots+u_pa_{i_p},$ where $a_{i_1},a_{i_2},\ldots ,a_{i_p}\in A$ and $u_1,u_2,\ldots, u_p$ are positive integers. By using the above notation, for each $x \in \Lambda$, we define
$$L_x^{(p)}=\{\Sigma\in L^{(p)}:\Sigma=x \}.$$

\begin{lem}\label{l1}
With the notations defined above, we have
\[
B(x)\supseteq \{v_1a_{i_1}+v_2a_{i_2}+\dots+v_pa_{i_p}: 0\leq v_i \leq u_i, \quad a_{i_1},a_{i_2},\ldots,a_{i_p}\in A \}
\]
and
\[
|B(x)|\geq \sum \limits_{i=1}^pu_i+\sum \limits_{1\leq i_1<i_2 \leq p}u_{i_1}u_{i_2}+\sum \limits_{1\leq i_1<i_2<i_3 \leq p}u_{i_1}u_{i_2}u_{i_3}+\ldots+u_1u_2\cdots u_p.
\]
\end{lem}
\begin{proof}
The proof of this Lemma follows from the definition of $B(x)$.
\end{proof}
The following Propositions provide us the bounds in term of $L_x^{(p)}$ to compute the graph $G_I(\Lambda)$ of order 6 and 7.
\begin{prop}\label{p1}
Let $\Lambda=<A>$ be a numerical semigroup of embedding dimension $n\geq 2$. Then $\mid G_I(\Lambda)\mid \neq 6$, if one of the following holds:
\begin{enumerate}
   \item $ L^{(p)}_x \neq \emptyset$ for $p\geq 3$
   \item $ \mid L^{(1)}_x \mid \geq3$.
  \item $L^{(2)}_x \neq \emptyset$ and  $L^{(1)}_x =\emptyset$.
  \item $ \mid L^{(1)}_x\mid \geq2$ and $ \mid L^{(2)}_x \mid\geq 2$.
\end{enumerate}
\end{prop}

\begin{prop}\label{p2}
Let $\Lambda=<A>$ be a numerical semigroup of embedding dimension $n\geq 2$. Then $\mid G_I(\Lambda)\mid \neq 7$, if one of the following hold:
\begin{enumerate}
  \item $ L^{(p)}_x \neq \emptyset$ for $p\geq 4$.
  \item $ \mid L^{(3)}_x\mid \geq2$.
  \item $ \mid L^{(3)}_x\mid =1$ and $ L^{(1)}_x \neq \emptyset$ or $ L^{(2)}_x \neq \emptyset$.
  \item $ \mid L^{(2)}_x\mid \geq4$.
  \item $ \mid L^{(2)}_x\mid =3$ and $L^{(1)}_x \neq \emptyset$.
  \item $ \mid L^{(2)}_x\mid =2$ and $\mid L^{(1)}_x\mid \geq 2$.
  \item $ \mid L^{(2)}_x\mid =1$ and $\mid L^{(1)}_x\mid \geq 2$.
  \item $ \mid L^{(1)}_x \mid \geq3$.

\end{enumerate}
\end{prop}
In the following, we give the proof of Proposition \ref{p1}, Proposition \ref{p2} can be proved in a similar way.
\begin{proof}
$(1):$ If $ L^{(p)}_x \neq \emptyset$ for $p\geq 3$ then there exist $a_i, a_j, a_k \in A$ such that $x=u_1a_i+u_2a_j+u_3a_k$, where $u_1, u_2, u_3 \geq 1$. By Lemma \ref{l1}, $a_i, a_j, a_k, a_i+a_j, a_i+a_k, a_j+a_k, x$ belongs to $ B^*(x)$. Clearly all these elements are distinct and hence  $\mid G_I(\Lambda)\mid \geq 7$.

$(2):$ If $ \mid L^{(1)}_x \mid \geq3$ then $x=u_1a_i$, $x=u_2a_j$ and $x=u_3a_k$, where $u_1, u_2, u_3 \geq 2$. Assume that $a_i<a_j<a_k$, then $u_1> u_2> u_3$. It follows that $u_1\geq 5$. By Lemma \ref{l1}, $ B^*(x)$ contains $a_i, 2a_i, 3a_i, 4a_i, 5a_i, a_j, a_k$ and therefore $\mid G_I(\Lambda)\mid \geq 7$.

$(3):$ If $ \mid L^{(2)}_x \mid =1$ then  $x=u_1a_i+v_1a_j$, where $u_1,v_1\geq 1$. If $u_1=v_1=1$ then $x=a_i+a_j$ and Lemma \ref{l1} gives $B^*(x)=\{a_i,a_j,a_i+a_j\}$. For $u_1=1$ and $v_1=2$, we have $x=a_i+2a_j$ and $B^*(x)=\{a_i,a_j,2a_j,a_i+a_j, a_i+2a_j\}$. Similarly, for $u_1=2$ and $v_1=1$  we get $B^*(x)=\{a_i,a_j,2a_i,a_i+a_j, 2a_i+a_j\}$. Now if $u_1=1$ and $v_1\geq3$ then $a_i, a_j, 2a_j, 3a_j, a_i+a_j, a_i+2a_j, a_i+3a_j$ are distinct elements of $B^*(x)$. Similarly, for $u_1\geq3$ and $v_1=1$ we get $a_j, a_i, 2a_i, 3a_i, a_i+a_j, 2a_i+a_j, 3a_i+a_j\in B^*(x)$. If $u_1,v_1\geq 2$ then note that $a_i, a_j, 2a_i, 2a_j, a_i+2a_j, 2a_i+a_j, a_i+a_j, 2a_i+2a_j$ are all distinct elements of $B^*(x)$. Hence, in all cases, we obtain $\mid G_I(\Lambda)\mid \neq 6$.

Now if $ \mid L^{(2)}_x \mid =2$ then $x=u_1a_i+v_1a_j$ and $x=u_2a_k+v_2a_l$, where $u_1, v_1, u_2, v_2\geq 1$ . It follows from the previous discussion that if $u_1=1$ and $v_1\geq3$, $u_1\geq3$ and $v_1=1$, $u_1,v_1\geq 2$, $u_2=1$ and $v_2\geq3$, $u_2\geq3$ and $v_2=1$, $u_2,v_2\geq 2$ then $\mid G_I(\Lambda)\mid \geq 7 $. If $a_i, a_j, a_k$ and $a_l$ are not distinct then one can easily check that the remaining possibilities does not holds. Now consider $a_i, a_j, a_k$ and $a_l$ are distinct. If $u_1=v_1=1$ and $u_2=v_2=1$ then $x=a_i+a_j=a_k+a_l$ and $B^*(x)=\{a_i, a_j, a_k, a_l, a_i+a_j\}$, hence $\mid G_I(\Lambda)\mid <6$. Now if $u_1=v_1=1$ and $u_2=2$, $v_2=1$ then $x=a_i+a_j=2a_k+a_l$. This gives $a_i, a_j, a_k, a_l, a_k+a_l, 2a_k, a_i+a_j$  distinct elements of $B^*(x)$ and therefore $\mid G_I(\Lambda)\mid \geq 7$. Similarly we get $\mid G_I(\Lambda)\mid \geq 7$, for all the remaining possibilities.

Finally, if $ \mid L^{(2)}_x \mid \geq 3$ then  $x=u_1a_i+v_1a_j$, $x=u_2a_k+v_2a_l$ and $x=u_3a_m+v_3a_q$, where $u_1, v_1, u_2, v_2, u_3, v_3\geq 1$. This give $a_i, a_j, a_k, a_l, a_m, a_q$ and  $a_i+a_j$ belongs to $B^*(x)$ and therefore $\mid G_I(\Lambda)\mid \geq 7$.

$(3):$ If $ \mid L^{(1)}_x\mid \geq2$ and $ \mid L^{(2)}_x \mid\geq 2$, then we can assume $x=u_1a_i+v_1a_j$, $x=u_2a_k+v_2a_l$, $x=ua_m$ and $x=va_q$, where
$u_1, v_1, u_2, v_2\geq 1$, $u, v\geq 2$. From case $(2)$, we have $x=a_i+a_j$ and $x=a_k+a_l$ is the only possibility for which $\mid G_I(\Lambda)\mid < 6$. Now if $a_m< a_q$ with $u>v$, then $x=a_i+a_j$ and $x=a_k+a_l$ gives $a_i, a_j, a_k, a_l, a_m, a_q, 2a_m, a_i+a_j$   distinct elements of $B^*(x)$. Hence $\mid G_I(\Lambda)\mid \geq 8$.
\end{proof}
In the proof of Proposition \ref{p1}, one can observe that there are some cases, where we have $\mid G_I(\Lambda)\mid < 6$. Therefore it is possible to get $\mid G_I(\Lambda)\mid =6$ by adding some suitable conditions. The similar is the case for Proposition \ref{p2}. In the next theorems, we classify all irreducible ideals $I$ such that the graph $G_I(\Lambda)$ is  of order $6$ or order $7$.
\begin{thm}\label{t1}
Let $\Lambda=<A>$ be a numerical semigroup of embedding dimension $n\geq 2$. If $\mid G_I(\Lambda)\mid = 6$ then $x\in \Lambda$ is one of the following:
\begin{enumerate}
  \item  $x=6a_i$.
  \item  $x=4a_i=3a_j$.
  \item  $x=5a_i=2a_j$.
  \item  $x=4a_i$ and $x=a_j+a_k$.
  \item  $x=4a_i$ and $3a_i=2a_k$.
  \item  $x=2a_i$ and $x=2a_k+a_l$.
  \item  $x=3a_i$, $x=2a_j$ and $x=a_k+a_l$.
  \item  $x=2a_i$, $x=a_k+a_l$ and $x=a_m+a_n$.
\end{enumerate}
\end{thm}
\begin{proof}
Given that $\mid G_I(\Lambda)\mid = 6$, then from Proposition \ref{p1}, it follows that $x\in \Lambda$ satisfy one of the following condition:
\begin{itemize}
  \item $ \mid L^{(1)}_x \mid =1$ and  $  L^{(p)}_x  =\emptyset$, $\forall$ $p\geq 2$,
  \item $ \mid L^{(1)}_x \mid =2$ and  $  L^{(p)}_x  =\emptyset$, $\forall$ $p\geq 2$,
  \item $ \mid L^{(1)}_x \mid =1$, $\mid L^{(2)}_x \mid =1$ and  $  L^{(p)}_x  =\emptyset$, $\forall$ $p\geq 3$,
  \item $ \mid L^{(1)}_x \mid =2$,  $\mid L^{(2)}_x \mid =1$ and  $  L^{(p)}_x  =\emptyset$, $\forall$ $p\geq 3$,
  \item $ \mid L^{(1)}_x \mid =1$,  $\mid L^{(2)}_x \mid =2$ and  $  L^{(p)}_x  =\emptyset$, $\forall$ $p\geq 3$.
\end{itemize}

$\mathbf{Case-1:}$ If $ \mid L^{(1)}_x \mid =1$ and  $  L^{(p)}_x  =\emptyset$, $\forall$ $p\geq 2$ then $x=la_i$, where $l\geq 1$. By Lemma \ref{l1}, we get $B^*(x)=\{a_i, 2a_i, \dots, la_i\}$. As $\mid G_I(\Lambda)\mid = 6$, it follows that $l=6$.

$\mathbf{Case-2:}$ if $ \mid L^{(1)}_x \mid =2$ and  $  L^{(p)}_x  =\emptyset$, $\forall$ $p\geq 2$ then $x=ua_i$ and $x=va_j$, where $a_i< a_j$, $1<v<u$ and $u$ is not a multiple of $v$. Then it follows from Lemma \ref{l1}, we have $B^*(x)=\{a_i, 2a_i,\dots, ua_i, a_j, 2a_j \dots, (v-1)a_j \}$. Let $pa_i=qa_j$ for some $q<p<u$ with $q=2,3,\dots, v-1$. Then $pa_i+(u-p)a_i=ua_i$ and we get $qa_j+(u-p)a_i=x.$ This give $ L^{(2)}_x  \neq \emptyset$, a contradiction. Therefore $a_i, 2a_i,\dots, ua_i, a_j, 2a_j \dots, (v-1)a_j$ are distinct elements of $B^*(x)$. As $\mid G_I(\Lambda)\mid = 6$, therefore either $u=4$ and $v=3$ or $u=5$ and $v=2$. This gives the case $(2)$ and case $(3)$.

$\mathbf{Case-3:}$ Let $ \mid L^{(1)}_x \mid =1=\mid L^{(2)}_x \mid$, then  $x=ra_i$  and  $x= ua_k+va_l$, where $r\geq 2$ and $u, v\geq 1$. By Proposition \ref{p1}$(2)$, we have $x=a_k+a_l$ or $x=2a_k+a_l$ or $x=a_k+2a_l$. If $x=a_k+a_l$ and $x=ra_i$, it is easy to observe that $a_i, a_k, a_l$ are distinct elements of $A$ and $ B^*(x)=\{a_k, a_l, a_i, 2a_i, \dots, (r-1)a_i, x\}$. then $\mid B^*(x)\mid = 6$ gives $r=4$ and we get the case $(4)$.

Now if $x=2a_k+a_l$ and $x=ra_i$ then $B^*(x)=\{a_k, a_l, 2a_k, a_k+a_l, a_i, 2a_i, \dots,$ $x=ra_i\}$. Note that $a_i=a_k$ is not possible, otherwise we get $a_l$ is a multiple of $a_i$. If $a_i=a_l$ then $a_i<a_k$, otherwise we get either $a_i=2a_k$ or $2a_i=2a_k$. These both conditions are not possible. Now $a_i<a_k$ gives $r>3$. $\mid B^*(x)\mid = 6$ gives $r=4$. Hence we get $x=4a_i$ and $3a_i=2a_k$ which is the case $(5)$.
Now if $a_i, a_k, a_l$ are distinct then $\mid B^*(x)\mid =r+4$. As $\mid B^*(x)\mid = 6$ therefore $r=2$ and we get the case $(6)$. Also if $x=a_k+2a_l$ and $x=ra_i$ then again we get the cases $(5)$ and $(6)$.

$\mathbf{Case-4:}$ If $\mid L^{(1)}_x \mid =2$ and  $\mid L^{(2)}_x \mid =1$, then $x=ra_i$, $x=sa_j$ and $x= ua_k+va_l$, where $r,s\geq 2$ and $u,v\geq 1$. Assume that $a_i<a_j$ then $1<s<r$ and $r$ is not a multiple of $s$. By Proposition \ref{p1}(2), we have $x=a_k+a_l$ or $x=2a_k+a_l$ or $x=a_k+2a_l$. For $u=2$ and $v=1$, we get $B^*(x)=\{a_k, a_l, 2a_k, a_k+a_l, a_j, 2a_j, \dots, (s-1)a_j, a_i, 2a_i, \dots, ra_i\}$. Similarly, for $u=1$ and $v=2$ we get $B^*(x)=\{a_k, a_l, 2a_l, a_k+a_l, a_j, 2a_j, \dots, (s-1)a_j, a_i, 2a_i, \dots, ra_i\}$. Note that $a_i$ or $a_j$ can not be equal to $a_k$, otherwise we get $a_l$ is a multiple of $a_i$ or $a_j$. Now if $a_i=a_l$ or $a_j=a_l$ then  $\mid B^*(x)\mid = r+s+2$ and if $a_i, a_j, a_k, a_l$  are different then $\mid B^*(x)\mid = r+s+3$, as  $1<s<r$ therefore $\mid B^*(x)\mid >6$. Now if $u=1,v=1$ then  $a_i, a_j, a_k, a_l$  are different and $B^*(x)=\{a_k, a_l,  a_j, 2a_j, \dots, (s-1)a_j, a_i, 2a_i, \dots, ra_i \}$. This gives $\mid B^*(x)\mid = r+s+1$. $\mid B^*(x)\mid = 6$ gives $r=3$ and $s=2$ which is the case $(7)$.

$\mathbf{Case-5:}$ If $\mid L^{(1)}_x \mid =1$ and  $  \mid L^{(2)}_x \mid =2$, then $x=u_1a_k+v_1a_l$, $x=u_2a_m+v_2a_n$ and $x= r a_i$, where $r\geq 2$ and $u_1, u_2, v_1, v_2\geq 1$.  By Proposition \ref{p1}(2), we have $x=a_k+a_l$ and $x=a_m+a_n$. In this case we have $ B^*(x)=\{a_k, a_l, a_m, a_n, a_i, 2a_i, \dots, r a_i\}$. Since $a_k, a_l, a_m, a_n, a_i$  are distinct, it follows that $\mid B^*(x)\mid = r+4$. $\mid B^*(x)\mid = 6$ gives $r=2$ which is  the case $(8)$.
\end{proof}
\begin{thm}\label{t2}
Let $\Lambda=<A>$ be a numerical semigroup of embedding dimension $n\geq 2$. If $\mid G_I(\Lambda)\mid = 7$ then $x\in \Lambda$ satisfy one of the following:
\begin{enumerate}
  \item  $x=7a_i$.
  \item  $x=5a_i=3a_j$.
  \item  $x=3a_i+a_j$ or $x=a_i+3a_j$.
  \item  $x=5a_i$ and $x=a_j+a_k$.
  \item  $x=3a_i$ and $x=2a_j+a_k$ or $x=a_j+2a_k$.
  \item  $x=a_i+a_j$ and $x=2a_k+a_l$.
  \item  $x=3a_i$ and $x=a_k+a_l$ and $x=a_m+a_n$.
  \item  $x=a_i+a_j$ and $x=a_k+a_l$ and $x=a_m+a_n$.
  \item  $x=a_i+a_j+a_k$.
\end{enumerate}
\end{thm}
\begin{proof}
Given that $\mid G_I(\Lambda)\mid = 7$, then from Proposition \ref{p2}, it follows that $x$ satisfy one of the following condition:
\begin{itemize}
  \item $ \mid L^{(1)}_x \mid =1$ and  $  L^{(p)}_x  =\emptyset$, $\forall$ $p\geq 2$,
  \item $ \mid L^{(1)}_x \mid =2$ and  $  L^{(p)}_x  =\emptyset$, $\forall$ $p\geq 2$,
  \item $L^{(1)}_x  =\emptyset,$, $\mid L^{(2)}_x \mid =1$, $L^{(p)}_x  =\emptyset,$  $\forall$ $p\geq 3$,
  \item $\mid L^{(1)}_x \mid =1$, $\mid L^{(2)}_x \mid =1$,  $L^{(p)}_x  =\emptyset$, $\forall$ $p\geq 3$,
  \item $\mid L^{(1)}_x \mid =1$, $\mid L^{(2)}_x \mid =1$,  $L^{(p)}_x  =\emptyset$, $\forall$ $p\geq 3$,
  \item $L^{(1)}_x  =\emptyset,$, $\mid L^{(2)}_x \mid =2$,  $L^{(p)}_x  =\emptyset$, $\forall$ $p\geq 3$,
  \item $\mid L^{(1)}_x \mid =1$, $\mid L^{(2)}_x \mid =2$,  $L^{(p)}_x  =\emptyset$, $\forall$ $p\geq 3$,
  \item $L^{(1)}_x  =\emptyset$, $\mid L^{(2)}_x \mid =3$,  $L^{(p)}_x  =\emptyset$, $\forall$ $p\geq 3$,
  \item $\mid L^{(3)}_x \mid =1$, $L^{(p)}_x  =\emptyset$, $\forall$ $p\neq 3$ and $x=a_i+a_j+a_k$.
\end{itemize}
These possibilities can be checked in a similar way as we did in Theorem \ref{t1} to get the required result.
\end{proof}

Theorem \ref{t1} and Theorem \ref{t2} give us all irreducible ideals $I$ such that $G_I(\Lambda)$ is of order 6 or order 7. This is easy to see that there are some cases where we get the graphs $G_I(\Lambda)$ which are isomorphic to each other. Our next propositions classify all graphs $G_I(\Lambda)$ upto isomorphism of order 6 and order 7.
\begin{cor}
A graph $G_I(\Lambda)$ of order $6$ is isomorphic to one of the graphs given in the Table $1$:
\begin{table}[hbt]
\begin{tabular}{|c|c|c|c|}
Type & \text{Degree sequence}  &  Graph  \\\hline
 $1$	& $(1, 2, 3, 3, 4, 5)$ & \setlength{\unitlength}{0.4cm}
             \begin{picture}(12,6)
              \put(5,5){\circle*{0.2}}
              \put(4.5,5){}
              \put(7,5){\circle*{0.2}}
\put(7.2,5){}
\put(5,1){\circle*{0.2}}
\put(4.3,1){}
\put(7,1){\circle*{0.2}}
\put(7.2,1){}
\put(4,3){\circle*{0.2}}
\put(3.3,3){}
\put(8,3){\circle*{0.2}}
\put(8.2,3){}
\put(5,5){\line(1,0){2}}
\put(5,1){\line(1,0){2}}
\put(4,3){\line(1,0){4}}
\put(4,3){\line(3,2){3}}
\put(5,1){\line(3,2){3}}
\put(5,1){\line(1,2){2}}
\put(7,1){\line(1,2){1}}
\put(8,3){\line(-1,2){1}}
\put(7,1){\line(0,1){4}}
\put(5.5,0){}
\end{picture}
\setlength{\unitlength}{0.4cm} \\\hline

$2$	& $(2, 3, 3, 4, 5, 5)$ & \setlength{\unitlength}{0.4cm}
\begin{picture}(12,6)
\put(5,5){\circle*{0.2}}
\put(4.5,5){}
\put(7,5){\circle*{0.2}}
\put(7.2,5){}
\put(5,1){\circle*{0.2}}
\put(4.3,1){}
\put(7,1){\circle*{0.2}}
\put(7.2,1){}
\put(4,3){\circle*{0.2}}
\put(3.3,3){}
\put(8,3){\circle*{0.2}}
\put(8.2,3){}
\put(5,5){\line(1,0){2}}
\put(4,3){\line(1,2){1}}
\put(4,3){\line(1,-2){1}}
\put(4,3){\line(1,0){4}}
\put(4,3){\line(3,2){3}}
\put(5,1){\line(3,2){3}}
\put(5,1){\line(1,2){2}}
\put(7,1){\line(1,2){1}}
\put(8,3){\line(-1,2){1}}
\put(7,1){\line(0,1){4}}
\put(4,3){\line(3,-2){3}}
\put(5.5,0){}
\end{picture}
\setlength{\unitlength}{0.4cm} \\\hline

$3$	& $(2, 3, 4, 4, 4, 5)$ & \setlength{\unitlength}{0.4cm}
\begin{picture}(12,6)
\put(5,5){\circle*{0.2}}
\put(4.5,5){}
\put(7,5){\circle*{0.2}}
\put(7.2,5){}
\put(5,1){\circle*{0.2}}
\put(4.3,1){}
\put(7,1){\circle*{0.2}}
\put(7.2,1){}
\put(4,3){\circle*{0.2}}
\put(2.6,3){}
\put(8,3){\circle*{0.2}}
\put(8.2,3){}
\put(5,5){\line(1,0){2}}
\put(5,1){\line(1,0){2}}
\put(4,3){\line(3,2){3}}
\put(5,1){\line(3,2){3}}
\put(5,1){\line(1,2){2}}
\put(7,1){\line(1,2){1}}
\put(8,3){\line(-1,2){1}}
\put(7,1){\line(0,1){4}}
\put(4,3){\line(3,-2){3}}
\put(4,3){\line(1,-2){1}}
\put(5,5){\line(-1,-2){1}}
\put(5.5,0){}
\end{picture}
\setlength{\unitlength}{0.4cm} \\\hline

$4$	& $(3, 3, 4, 4, 5, 5)$ & \setlength{\unitlength}{0.4cm}
\begin{picture}(12,6)
\put(5,5){\circle*{0.2}}
\put(4.5,5){}
\put(7,5){\circle*{0.2}}
\put(7.2,5){}
\put(5,1){\circle*{0.2}}
\put(4.4,1){}
\put(7,1){\circle*{0.2}}
\put(7.2,1){}
\put(4,3){\circle*{0.2}}
\put(3.3,3){}
\put(8,3){\circle*{0.2}}
\put(8.2,3){}
\put(5,5){\line(1,0){2}}
\put(4,3){\line(3,2){3}}
\put(5,1){\line(3,2){3}}
\put(5,1){\line(1,2){2}}
\put(7,1){\line(1,2){1}}
\put(8,3){\line(-1,2){1}}
\put(7,1){\line(0,1){4}}
\put(4,3){\line(3,-2){3}}
\put(5,5){\line(3,-2){3}}
\put(5,5){\line(0,-1){4}}
\put(5,5){\line(-1,-2){1}}
\put(5,5){\line(1,-2){2}}
\put(5.5,0){}
\end{picture}
\setlength{\unitlength}{0.4cm} \\\hline

$5$	& $(3, 4, 4, 4, 4, 5)$  & \setlength{\unitlength}{0.4cm}
\begin{picture}(12,6)
\put(5,5){\circle*{0.2}}
\put(4.5,5){}
\put(7,5){\circle*{0.2}}
\put(7.2,5){}
\put(5,1){\circle*{0.2}}
\put(4.3,1){}
\put(7,1){\circle*{0.2}}
\put(7.2,1){}
\put(4,3){\circle*{0.2}}
\put(3.3,3){}
\put(8,3){\circle*{0.2}}
\put(8.2,3){}
\put(5,5){\line(1,0){2}}
\put(4,3){\line(1,2){1}}
\put(4,3){\line(1,-2){1}}
\put(5,1){\line(1,0){2}}
\put(4,3){\line(1,0){4}}
\put(4,3){\line(3,2){3}}
\put(5,1){\line(3,2){3}}
\put(5,1){\line(1,2){2}}
\put(7,1){\line(1,2){1}}
\put(8,3){\line(-1,2){1}}
\put(7,1){\line(0,1){4}}
\put(5,5){\line(1,-2){2}}
\put(5.5,0){}
\end{picture}
\setlength{\unitlength}{0.4cm} \\\hline

$6$	&$(4, 4, 4, 4, 5, 5)$  & \setlength{\unitlength}{0.4cm}
\begin{picture}(12,6)
\put(5,5){\circle*{0.2}}
\put(4.5,5){}
\put(7,5){\circle*{0.2}}
\put(7.2,5){}
\put(5,1){\circle*{0.2}}
\put(4.3,1){}
\put(7,1){\circle*{0.2}}
\put(7.2,1){}
\put(4,3){\circle*{0.2}}
\put(3.3,3){}
\put(8,3){\circle*{0.2}}
\put(8.2,3){}
\put(5,5){\line(1,0){2}}
\put(4,3){\line(1,0){4}}
\put(4,3){\line(3,2){3}}
\put(5,1){\line(3,2){3}}
\put(5,1){\line(1,2){2}}
\put(8,3){\line(-1,2){1}}
\put(7,1){\line(0,1){4}}
\put(4,3){\line(1,-2){1}}
\put(8,3){\line(-1,-2){1}}
\put(5,5){\line(-1,-2){1}}
\put(5,5){\line(1,-2){2}}
\put(5,1){\line(1,0){2}}
\put(5,5){\line(3,-2){3}}
\put(5.5,0){}
\end{picture}
\setlength{\unitlength}{0.4cm}
\end{tabular}
\caption{ }
\label{1}
\end{table}
\end{cor}
\begin{proof}
This is an easy consequence of Theorem \ref{t1}.
\end{proof}

\begin{cor}
A graph $G_I(\Lambda)$ of order $7$ is isomorphic to one of the graphs given in the Table $2$:
\begin{table}[hbt]
\begin{tabular}{|c|c|c|c|}
Type & \text{Degree sequence}  &  Graph  \\\hline
 $1$	& $(1, 2, 3, 3, 4, 5, 6)$ & \setlength{\unitlength}{0.4cm}
             \begin{picture}(12,6)
\put(6,5){\circle*{0.2}}
\put(5.3,5){}
\put(8,4){\circle*{0.2}}
\put(8.2,4){}
\put(4,4){\circle*{0.2}}
\put(3.4,4){}
\put(5,2){\circle*{0.2}}
\put(4.3,2){}
\put(7,2){\circle*{0.2}}
\put(7.3,2){}
\put(4,3){\circle*{0.2}}
\put(3.3,3){}
\put(8,3){\circle*{0.2}}
\put(8.2,3){}
\put(6,5){\line(-2,-1){2}}
\put(6,5){\line(-1,-1){2}}
\put(6,5){\line(-1,-3){1}}
\put(6,5){\line(1,-3){1}}
\put(6,5){\line(1,-1){2}}
\put(6,5){\line(2,-1){2}}
\put(4,3){\line(4,1){4}}
\put(5,2){\line(3,1){3}}
\put(5,2){\line(3,2){3}}
\put(7,2){\line(1,1){1}}
\put(7,2){\line(1,2){1}}
\put(8,3){\line(0,1){1}}
\put(5.5,1){}
\end{picture}
\setlength{\unitlength}{0.4cm} \\\hline

$2$	& $(2, 3, 3, 4, 5, 5, 6)$ & \setlength{\unitlength}{0.4cm}

\begin{picture}(12,7)
\put(5,6){\circle*{0.2}}
\put(4,6.2){}
\put(7,6){\circle*{0.2}}
\put(7,6.2){}
\put(3,4){\circle*{0.2}}
\put(2.6,4){}
\put(5,2){\circle*{0.2}}
\put(5,1.6){}
\put(7,2){\circle*{0.2}}
\put(7,1.5){}
\put(9,3){\circle*{0.2}}
\put(9.2,3){}
\put(9,5){\circle*{0.2}}
\put(9.2,5){}
\put(3,4){\line(1,-1){2}}
\put(3,4){\line(2,-1){4}}
\put(3,4){\line(6,-1){6}}
\put(3,4){\line(6,1){6}}
\put(3,4){\line(2,1){4}}
\put(3,4){\line(1,1){2}}
\put(5,2){\line(4,1){4}}
\put(7,2){\line(2,1){2}}
\put(7,2){\line(-1,2){2}}
\put(9,3){\line(-2,3){2}}
\put(9,3){\line(-4,3){4}}
\put(9,5){\line(-2,1){2}}
\put(9,5){\line(-4,1){4}}
\put(7,6){\line(-1,0){2}}
\put(5.5,1){}
\end{picture}
\setlength{\unitlength}{0.4cm} \\\hline

$3$	&$(3, 3, 3, 5, 5, 5, 6)$  & \setlength{\unitlength}{0.4cm}
\begin{picture}(12,7)
\put(5,6){\circle*{0.2}}
\put(4.3,6.2){}
\put(7,6){\circle*{0.2}}
\put(7.2,6){}
\put(3,4){\circle*{0.2}}
\put(2.6,4){}
\put(5,2){\circle*{0.2}}
\put(5,1.6){}
\put(7,2){\circle*{0.2}}
\put(7,1.5){}
\put(9,3){\circle*{0.2}}
\put(9.2,3){}
\put(9,5){\circle*{0.2}}
\put(9.2,5){}
\put(3,4){\line(1,-1){2}}
\put(3,4){\line(2,-1){4}}
\put(3,4){\line(6,-1){6}}
\put(3,4){\line(6,1){6}}
\put(3,4){\line(2,1){4}}
\put(3,4){\line(1,1){2}}
\put(5,2){\line(4,3){4}}
\put(5,2){\line(1,2){2}}
\put(7,2){\line(2,3){2}}
\put(7,2){\line(-1,2){2}}
\put(9,3){\line(-2,3){2}}
\put(9,3){\line(-4,3){4}}
\put(9,5){\line(-2,1){2}}
\put(9,5){\line(-4,1){4}}
\put(7,6){\line(-1,0){2}}
\put(5.5,1){}
\end{picture}
\setlength{\unitlength}{0.4cm} \\\hline

$4$	& $(3, 4, 4, 5, 5, 5, 6)$  & \setlength{\unitlength}{0.4cm}
\begin{picture}(12,7)
\put(5,6){\circle*{0.2}}
\put(4.2,6){}
\put(7,6){\circle*{0.2}}
\put(7.2,6){}
\put(3,4){\circle*{0.2}}
\put(2.6,4){}
\put(5,2){\circle*{0.2}}
\put(5,1.6){}
\put(7,2){\circle*{0.2}}
\put(7,1.5){}
\put(9,3){\circle*{0.2}}
\put(9.2,3){}
\put(9,5){\circle*{0.2}}
\put(9.2,5){}
\put(3,4){\line(1,-1){2}}
\put(3,4){\line(2,-1){4}}
\put(3,4){\line(6,-1){6}}
\put(3,4){\line(6,1){6}}
\put(3,4){\line(2,1){4}}
\put(3,4){\line(1,1){2}}
\put(5,2){\line(1,2){2}}
\put(5,2){\line(0,1){4}}
\put(7,2){\line(2,3){2}}
\put(7,2){\line(0,1){4}}
\put(7,2){\line(-1,2){2}}
\put(9,3){\line(0,1){2}}
\put(9,3){\line(-2,3){2}}
\put(9,3){\line(-4,3){4}}
\put(9,5){\line(-2,1){2}}
\put(9,5){\line(-4,1){4}}
\put(5.5,1){}
\end{picture}
\setlength{\unitlength}{0.4cm} \\\hline

$5$	& $(4, 4, 5, 5, 5, 5, 6)$ & \setlength{\unitlength}{0.4cm}
\begin{picture}(12,7)
\put(5,6){\circle*{0.2}}
\put(4.3,6.2){}
\put(7,6){\circle*{0.2}}
\put(7.2,6){}
\put(3,4){\circle*{0.2}}
\put(2.6,4){}
\put(5,2){\circle*{0.2}}
\put(5,1.6){}
\put(7,2){\circle*{0.2}}
\put(7,1.5){}
\put(9,3){\circle*{0.2}}
\put(9.2,3){}
\put(9,5){\circle*{0.2}}
\put(9.2,5){}
\put(3,4){\line(1,-1){2}}
\put(3,4){\line(2,-1){4}}
\put(3,4){\line(6,-1){6}}
\put(3,4){\line(6,1){6}}
\put(3,4){\line(2,1){4}}
\put(3,4){\line(1,1){2}}
\put(5,2){\line(1,0){2}}
\put(5,2){\line(4,1){4}}
\put(5,2){\line(4,3){4}}
\put(7,2){\line(2,1){2}}
\put(7,2){\line(2,3){2}}
\put(7,2){\line(-1,2){2}}
\put(9,3){\line(-2,3){2}}
\put(9,3){\line(-4,3){4}}
\put(9,5){\line(-2,1){2}}
\put(9,5){\line(-4,1){4}}
\put(7,6){\line(-1,0){2}}
\put(5.5,1){}
\end{picture}
\setlength{\unitlength}{0.4cm} \\\hline

$6$	& $(5, 5, 5, 5, 5, 5, 6)$  & \setlength{\unitlength}{0.4cm}
\begin{picture}(12,7)
\put(5,6){\circle*{0.2}}
\put(4.4,6){}
\put(7,6){\circle*{0.2}}
\put(7.2,6){}
\put(3,4){\circle*{0.2}}
\put(2.6,4){}
\put(5,2){\circle*{0.2}}
\put(5,1.6){}
\put(7,2){\circle*{0.2}}
\put(7,1.5){}
\put(9,3){\circle*{0.2}}
\put(9.2,3){}
\put(9,5){\circle*{0.2}}
\put(9.2,5){}
\put(3,4){\line(1,-1){2}}
\put(3,4){\line(2,-1){4}}
\put(3,4){\line(6,-1){6}}
\put(3,4){\line(6,1){6}}
\put(3,4){\line(2,1){4}}
\put(3,4){\line(1,1){2}}
\put(5,2){\line(4,1){4}}
\put(5,2){\line(4,3){4}}
\put(5,2){\line(1,2){2}}
\put(5,2){\line(0,1){4}}
\put(7,2){\line(2,1){2}}
\put(7,2){\line(2,3){2}}
\put(7,2){\line(0,1){4}}
\put(7,2){\line(-1,2){2}}
\put(9,3){\line(-2,3){2}}
\put(9,3){\line(-4,3){4}}
\put(9,5){\line(-2,1){2}}
\put(9,5){\line(-4,1){4}}
\put(5.5,1){}
\end{picture}
\setlength{\unitlength}{0.4cm}
\end{tabular}
\caption{ }
\label{2}
\end{table}
\end{cor}
\begin{proof}
This is an easy consequence of Theorem \ref{t2}.
\end{proof}
\begin{thm}
Let $\Lambda$ be a numerical semigroup of embedding dimension $n\geq2$. Then  $G_I(\Lambda)$ is planar if one of the following hold:
 \begin{enumerate}
          \item  $\mid G_I(\Lambda)\mid \leq 5$.
          \item  $\mid G_I(\Lambda)\mid = 6$ and $G_I(\Lambda)$ is of type $1, 2, 3$ or $4$.
          \item  $\mid G_I(\Lambda)\mid = 7$ and $G_I(\Lambda)$ is of type $1, 2$ or $3$.
\end{enumerate}
\end{thm}
\begin{proof}
$(1)$ is trivial.\\
$(2)$ If $\mid G_I(\Lambda)\mid = 6$ and $G_I(\Lambda)$ is of type $1, 2$ or $3$ then $G_I(\Lambda)$ is trivially planar. Now if $G_I(\Lambda)$ is of type $4$ then $B^*(x)=\{a_i, a_j, a_k, a_j+a_k, 2a_j, 2a_i+a_j, x\}$ with deg$(v_{a_i})=5$, deg$(v_{a_j})=3$, deg$(v_{a_k})=3$, deg$(v_{a_j+a_k})=4$, deg$(v_{2a_j})=4$, and deg$(v_{x})=5$.\\
Since $\mid G_I(\Lambda)\mid =6$, therefore $ cl(G_I(\Lambda))=4$. This shows that $G_I(\Lambda)$ cannot have a subgraph which is isomorphic to complete graph $K_5$. Now consider a subgraph $H$ of $G_I(\Lambda)$ such that $\mid H\mid = 6$ and the degree sequence of $H$ is $(3,3,3,3,3,3)$. This give $V(H)=V(G_I(\Lambda))$ and $E(H)=E(G_I(\Lambda))-\{v_{a_i}v_{x}, v_{a_i}v_{2a_j}, v_{a_j+a_k}v_{x}\}$ or $V(H)=V(G_I(\Lambda))$ and $E(H)=E(G_I(\Lambda))-\{v_{a_i}v_{x}, v_{a_i}v_{a_j+a_k}, v_{2a_j}v_{x}\}$. In both cases, one can easily see that $H\ncong K_{3,3}$.\\
$(3)$ can be proved in a similar way as we proved $(2)$.
\end{proof}
\begin{thm}
Let $\Lambda$ be a numerical semigroup of embedding dimension $n\geq2$. Then  $G_I(\Lambda)$ is non-planar if one of the following hold:

        \begin{enumerate}
   \item $\mid G_I(\Lambda)\mid \geq 8$,
   \item $\mid G_I(\Lambda)\mid =6 $ and $G_I(\Lambda)$ is of type $5$ or $6$.
   \item $\mid G_I(\Lambda)\mid =7 $ and $G_I(\Lambda)$ is of type $4, 5$ or $6$.
\end{enumerate}
\end{thm}
\begin{proof}
$\mathbf{(1):}$ If $\mid G_I(\Lambda)\mid \geq 8$ then $ cl(G_I(\Lambda))\geq 5$ and therefore $G_I(\Lambda)$ must has a subgraph isomorphic to complete graph $K_5$.\newline
$\mathbf{(2):}$ Assume that  $\mid G_I(\Lambda)\mid =6 $ and $G_I(\Lambda)$ is of type $5$ then $B^*(x)=\{a_i, a_j, 2a_i, 2a_j,$ $ 3a_i, 4a_i\}$ with deg$(v_{a_i})=3$, deg$(v_{a_j})=4$, deg$(v_{2a_i})=4$, deg$(v_{2a_j})=4$, deg$(v_{3a_i})=4$ and deg$(v_{4a_i})=5$. Since $\mid G_I(\Lambda)\mid =6$, therefore $ cl(G_I(\Lambda))=4$. This shows that $G_I(\Lambda)$ cannot have a subgraph which is isomorphic to complete graph $K_5$. Consider a subgraph $H$ of $G_I(\Lambda)$ such that $V(H)=V(G_I(\Lambda))$ and $E(H)=E(G_I(\Lambda))- \{v_{a_j}v_{4a_i}, v_{2a_i}v_{3a_i}, v_{2a_j}v_{4a_i}\}$. Note that we can partition the set of vertices of $H$ into $V_1=\{a_i, 2a_i, 3a_i\}$ and $V_2=\{a_j, 2a_j, 4a_i\}$ such that no edge has both endpoints in the same subset and every possible edge that could connect vertices in different subsets is part of the graph. This implies $H$ is isomorphic to complete bipartite graph $K_{3,3}$ and therefore $G_I(\Lambda)$ is non-planar.\newline
Remaining cases of $\mathbf{(2)}$ and $\mathbf{(3)}$ can be proved in a similar way.\\

\setlength{\unitlength}{1cm}
\begin{picture}(12,5)
\put(6,4){\circle*{0.2}}
\put(5.5,4){$v_{a_i}$}
\put(8,4){\circle*{0.2}}
\put(8.2,4){$v_{4a_i}$}
\put(6,1){\circle*{0.2}}
\put(5.5,0.9){$v_{2a_j}$}
\put(8,1){\circle*{0.2}}
\put(8.2,1){$v_{3a_i}$}
\put(5,3){\circle*{0.2}}
\put(4.5,3){$v_{a_j}$}
\put(5,2){\circle*{0.2}}
\put(4.5,2){$v_{2a_i}$}
\put(6,4){\line(-1,-1){1}}
\put(6,4){\line(0,-1){3}}
\put(6,4){\line(1,0){2}}
\put(5,3){\line(0,-1){1}}
\put(5,3){\line(3,1){3}}
\put(5,3){\line(3,-2){3}}
\put(5,2){\line(1,-1){1}}
\put(5,2){\line(3,-1){3}}
\put(5,2){\line(3,2){3}}
\put(6,1){\line(1,0){2}}
\put(6,1){\line(2,3){2}}
\put(8,1){\line(0,3){3}}
\put(5.5,0){Fig-1}
\end{picture}\\
\setlength{\unitlength}{1cm}
\begin{picture}(12,5)
\put(6,4){\circle*{0.2}}
\put(6,4.2){$v_{a_i}$}
\put(7,4){\circle*{0.2}}
\put(7,4.2){$v_{2a_i}$}
\put(8,4){\circle*{0.2}}
\put(8,4.2){$v_{3a_i}$}
\put(6,2){\circle*{0.2}}
\put(6,1.5){$v_{a_j}$}
\put(7,2){\circle*{0.2}}
\put(7,1.5){$v_{2a_j}$}
\put(8,2){\circle*{0.2}}
\put(8,1.5){$v_{4a_i}$}
\put(6,4){\line(1,-1){2}}
\put(6,4){\line(0,-1){2}}
\put(6,4){\line(1,-2){1}}
\put(7,4){\line(-1,-2){1}}
\put(7,4){\line(0,-1){2}}
\put(7,4){\line(1,-2){1}}
\put(8,4){\line(-1,-1){2}}
\put(8,4){\line(-1,-2){1}}
\put(8,4){\line(0,-1){2}}
\put(6,0.7){Fig-2}
\end{picture}
\end{proof}

\section{Conclusion}
In this article, we have given a complete answer about the planarity of the graph $G_I(\Lambda)$ associated with the irreducible ideal of a numerical semigroup. However, for any integral ideal $I$, this is an open question.

{\bf Conflict of Interests:}
The authors hereby declare that there is no conflict of interests regarding the publication of this paper.

{\bf Data Availability Statement:}
No data is required for this study.

{\bf Funding Statement:}
This research is carried out as a part of the employment of the authors.


\begin{thebibliography}{99}
\bibitem{1}	Afkhami, M. and Khashyarmanesh, K., The intersection graph of ideals of a lattice, Note Mat. 34(2) (2014), 135-143.

\bibitem{2} D.F. Anderson, P.S. Livingston, The zero-divisor graph of a commutative ring, J. Algebra 217 (1999) 434–447.

\bibitem{3} Anderson, D. D. and Badawi, A., The total graph of a commutative ring, J. Algebra 320 (2008), 2706-2719.

\bibitem{4} A. Badawi, On the annihilator graph of a commutative ring, Comm. Algebra 42 (2014) 108–121. http://dx.doi.org/10.1080/00927872.2012.707262.

\bibitem{5} M. Behboodi, Z. Rakeei, The annihilating-ideal graph of commutative rings I, J. Algebra Appl. 10 (4) (2011) 727–739.

\bibitem{6} H.R. Maimani, M. Salimi, A. Sattari, S. Yassemi, Comaximal graph of commutative rings, J. Algebra 319 (2008) 1801–1808.

\bibitem{7} D.F. Anderson, S.B. Mulay, On the diameter and girth of a zero-divisor graph, J. Pure Appl. Algebra 210 (2007) 543-550.

\bibitem{8} T. Tamizh Chelvam, K. Selvakumar, Central sets in annihilating-ideal graph of a commutative ring, J. Combin. Math. Combin. Comput. 88 (2014) 277–288.

\bibitem{9} T. Tamizh Chelvam, K. Selvakumar, Domination in the directed zero-divisor graph of ring of matrices, J. Combin. Math. Combin. Comput. 91 (2014) 155–163.

\bibitem{10} S. Akbari, H.R. Maimani, S. Yassemi, When a zero-divisor graph is planar or a complete r-partite graph, J. Algebra 270 (2003) 169-180.

\bibitem{11} Beck, I., Coloring of commutative rings, J. Algebra 116 (1998), 208-226.

\bibitem{12} Barucci, V., Decompositions of ideals into irreducible ideals in numerical semigroups. Journal of Commutative Algebra, 2(3) (2010), 281-294.

\bibitem{13} Binyamin, M. A., Siddiqui, H. M. A., Khan, N. M., Aslam, A., Rao, Y., Characterization of graphs associated with numerical semigroups. Mathematics, 7, 557 (2019),  doi:10.3390/math7060557.

\bibitem{13a} P. Xu, M. A. Binyamin, A. Aslam, W. Ali, H. Mahmood and Hao Zhou, Characterization of Graphs Associated to the Ideal of the Numerical Semigroups, Journal of Mathematics, Volume 2020, Article ID 6094372, 6 pages.

\bibitem{14} Chartrand, G., Introduction to graph theory. Tata McGraw-Hill Education (2006).

\bibitem{15} Diestel, R., Graph Theory, New York: Springer-Verlag (1997).

\bibitem{16} West, D. B., Introduction to graph theory (Vol. 2). Upper Saddle River: Prentice hall (2001).

%
%
%
%
%
%
%
%
%
%
%
%
%
%
%
%
%
%
		
		
			
		\end{thebibliography}
\end{document}